\documentclass[10pt,reqno]{amsart}
\usepackage{latexsym,amsmath,amssymb,amscd}
\usepackage{enumerate}

\def\today{\ifcase \month \or
   January \or February \or March \or April \or
   May \or June \or July \or August \or
   September \or October \or November \or December \fi
   \space\number\day , \number\year}
   
\makeatletter
  \newcommand\@dotsep{4.5}
  \def\@tocline#1#2#3#4#5#6#7{\relax
     \ifnum #1>\c@tocdepth 
     \else
     \par \addpenalty\@secpenalty\addvspace{#2}%
     \begingroup \hyphenpenalty\@M
     \@ifempty{#4}{%
     \@tempdima\csname r@tocindent\number#1\endcsname\relax
        }{%
         \@tempdima#4\relax
           }%
      \parindent\z@ \leftskip#3\relax \advance\leftskip\@tempdima\relax
      \rightskip\@pnumwidth plus1em \parfillskip-\@pnumwidth
       #5\leavevmode\hskip-\@tempdima #6\relax
       \leaders\hbox{$\m@th
       \mkern \@dotsep mu\hbox{.}\mkern \@dotsep mu$}\hfill
       \hbox to\@pnumwidth{\@tocpagenum{#7}}\par
       \nobreak
        \endgroup
         \fi}
\makeatother 

 \newtheorem{thm}{Theorem}[section]
 \newtheorem{cor}[thm]{Corollary}
 \newtheorem{lem}[thm]{Lemma}
 \newtheorem{prop}[thm]{Proposition}
 \theoremstyle{definition}
 \newtheorem{defn}[thm]{Definition}
 \theoremstyle{remark}
 \newtheorem{rem}[thm]{Remark}
 \newtheorem{ex}[thm]{Example}
 \numberwithin{equation}{section}
 
\newcommand{\1}{{\bf 1}}
\newcommand{\Ad}{{\rm Ad}}

\newcommand{\Ci}{{\mathcal C}^\infty}
\newcommand{\de}{{\rm d}}
\newcommand{\ee}{{\rm e}}
\newcommand{\End}{{\rm End}\,}

\newcommand{\ie}{{\rm i}}
\newcommand{\Ker}{{\rm Ker}\,}
\newcommand{\Ran}{{\rm Ran}\,}

\newcommand{\CC}{{\mathbb C}}
\newcommand{\HH}{{\mathbb H}}
\newcommand{\RR}{{\mathbb R}}
\newcommand{\TT}{{\mathbb T}}

\newcommand{\Ac}{{\mathcal A}}
\newcommand{\Bc}{{\mathcal B}}

\newcommand{\Fc}{{\mathcal F}}
\newcommand{\Hc}{{\mathcal H}}
\newcommand{\Oc}{{\mathcal O}}
\newcommand{\Vc}{{\mathcal V}}

\renewcommand{\gg}{{\mathfrak g}}
\newcommand{\hg}{{\mathfrak h}}

\begin{document}

\title[Representations of nilpotent Lie groups]
 {Representations of nilpotent Lie groups via measurable dynamical systems}

\author[I. Belti\c t\u a]{Ingrid Belti\c t\u a}

\address{%
Institute of Mathematics\\ 
``Simion Stoilow''\\ 
of the Romanian Academy\\
P.O. Box 1-764\\
Bucharest\\
Romania}

\email{ingrid.beltita@gmail.com
}

\thanks{This research has been partially supported by the Grant
of the Romanian National Authority for Scientific Research, CNCS-UEFISCDI,
project number PN-II-ID-PCE-2011-3-0131.}

\author[D. Belti\c t\u a]{Daniel Belti\c t\u a}

\address{%
Institute of Mathematics\\ 
``Simion Stoilow''\\ 
of the Romanian Academy\\
P.O. Box 1-764\\
Bucharest\\
Romania}

\email{beltita@gmail.com
}

\subjclass{Primary 22E27; Secondary 22E25, 22E66, 28D15}

\keywords{dynamical system, ergodic action, semidirect product}

\date{18 October 2015}

\begin{abstract}
We study unitary representations associated to cocycles of 
measurable dynamical systems. 
Our main result establishes conditions on a cocycle,  
ensuring that ergodicity of the dynamical system under consideration 
is equivalent to irreducibility of its corresponding unitary representation. 
This general result is applied to some representations of finite-dimensional nilpotent Lie groups 
and to some representations of infinite-dimensional Heisenberg groups. 
\end{abstract}

\maketitle

\section{Introduction}

A measurable dynamical system is a measure space $(X,\mu)$ 
endowed with a group action on the right 
$X\times S\to X$, $(x,s)\mapsto x.s$, 
for which the measure~$\mu$ is quasi-invariant, 
hence $\de\mu(x.s)=j(x,s)\de\mu(x)$ for a suitable a.e. defined positive measurable function $j(\cdot,s)$ on $X$.  
A scalar cocycle of this measurable dynamical system is a family $\{a(\cdot, s)\}_{s\in S}$ 
of a.e. defined measurable functions on $X$ with values in the unit circle~$\TT$, 
for which the map $\pi_a\colon S\to \Bc(L^2(X,\mu))$ is a unitary representation, where 
$$\pi_a(s)\colon L^2(X,\mu)\to L^2(X,\mu),\quad 
(\pi_a(s)\varphi)(x)=j(x,s)^{1/2}a(x,s)\varphi(x.s).$$
Our main abstract result is Theorem~\ref{quasi_th} which 
establishes conditions on the cocycle~$a$,  
ensuring that ergodicity of the above dynamical system  
is equivalent to irreducibility of the unitary representation~$\pi_a$. 
The unifying force of this result is then illustrated by a variety of applications, 
including unitary irreducible representations 
of finite-dimensional nilpotent Lie groups 
and some representations of infinite-dimensional Heisenberg groups. 

\subsection*{Some preliminaries on measure theory}

\begin{lem}\label{masa}
Let $(X,\mu)$ be any measure space and $\Hc:=L^2(X,\mu)$. 
For any $\psi\in L^\infty(X,\mu)$ let $M_\psi\in\Bc(\Hc)$ be the multiplication-by-$\psi$ operator, 
and define $\Ac:=\{M_\psi\mid \psi\in L^\infty(X,\mu)\}$. 
If at least one of the following conditions is satisfied:  
\begin{enumerate}
\item $X$ is a locally compact space and $\mu$ is a Radon measure; 
\item one has $\mu(X)<\infty$ and $\Hc$ is separable; 
\end{enumerate}
then $\Ac$ is a maximal abelian self-adjoint subalgebra of $\Bc(\Hc)$. 
\end{lem}

\begin{proof}
If the first condition is satisfied then the assertion follows by \cite[Ch.\ I, \S 7, no.\ 3, Th.\ 2]{Di69}. 
If the second condition is satisfied, then the constant function $1\in L^\infty(X,\mu)\subseteq L^2(X,\mu)$ 
is a cyclic vector for $\Ac$, hence the conclusion follows by \cite[Th. 2.3.4]{SS08}. 
\end{proof}

\begin{defn}\label{erg_def}
\normalfont
Let $\alpha\colon G\times X\to X$ be any group action by measurable transformations of a measure space $(X,\mu)$. 
The action $\alpha$ is called \emph{ergodic} if for every measurable set $A\subseteq X$ 
with $\mu((A\, \triangle \, \alpha_g(A) )=0$ for all $g\in G$, one has either $\mu(A)=0$ or $\mu(X\setminus A)=0$. 
(Here, for two sets $X$ and $Y$, $X\,  \triangle\,  Y$ denotes the symmetric difference $X \,  \triangle\, Y= (X\setminus Y)\cup (Y\setminus X)$.)
\end{defn}

\begin{rem}\label{erg_rem}
\normalfont
In the framework of Definition~\ref{erg_def} it is straightforward to check that 
the group action~$\alpha$ is ergodic if and only if the equivalence classes of a.e. 
constant functions in $L^\infty(X,\mu)$ are the only elements 
$\varphi\in L^\infty(X,\mu)$ 
with $\varphi\circ\alpha_g=\varphi$ for all $g\in G$. 
This implies that if $\alpha$ is a transitive action 
(or more generally, if for every $x\in X$ with its orbit $G.x:=\{\alpha_g(x)\mid g\in G\}$ 
one has $\mu(X\setminus G.x)=0$), 
then $\alpha$ is ergodic. 
 
We refer to \cite{Ta03} for the role of ergodic actions in the theory of operator algebras. 
\end{rem}

\section{General results}

To begin with, we recall some ideas from \cite[Ch. I, Subsect.\ 1.4]{Is96}. 

\begin{defn}\label{quasi_def}
\normalfont
A \emph{measurable dynamical system} consists of a measure space $(X,\mu)$ 
endowed with a group action on the right 
$$\beta\colon X\times S\to X, \quad (x,s)\mapsto\beta_s(x)=:x.s,$$ 
for which the measure~$\mu$ is quasi-invariant. 
Then for every $s\in S$ there is an a.e. defined positive function $j(\cdot,s)$ on $X$ 
for which $(\beta_s)_*(\mu)=j(\cdot,s)\mu$, 
where $(\beta_s)_*(\mu)$ denotes the pushforward of the measure~$\mu$ through the map~$\beta_s$. 
Hence for every measurable set $E\subseteq X$ one has 
$$\mu(\beta_s(E))=\int\limits_E j(x,s)\de\mu(x). $$
A \emph{scalar cocycle} of this measurable dynamical system is a family $\{a(\cdot, s)\}_{s\in S}$ 
consisting of a.e. defined measurable functions on $X$ with values in the unit circle~$\TT$, 
satisfying the conditions 
$$a(x,s_1s_2)=a(x,s_1)a(x.s_1,s_2)\text{ and }a(x,\1)=x$$
for a.e. $x\in X$ and all $s_1,s_2\in S$. 

In the above setting we also define $\Hc:=L^2(X,\mu)$ 
and for every $s\in S$, 
$$\pi_a(s)\colon \Hc\to\Hc,\quad 
\pi_a(s)\varphi=j(\cdot,s)^{1/2}a(\cdot,s)(\varphi\circ\beta_s)(\cdot).$$
\end{defn}

\begin{rem}
\normalfont
In Definition~\ref{quasi_def}, since $(\beta_{s_1s_2})_*(\mu)=(\beta_{s_1})_*((\beta_{s_2})_*(\mu))$ 
for all $s_1,s_2\in S$, it is easily checked that the family $\{a(\cdot, s)\}_{s\in S}$ 
satisfies the conditions of a scalar cocycle, except that the functions from this family 
take values in the multiplicative group $(0,\infty)$ instead of the unit circle~$\TT$. 
\end{rem}

The following result is a special case of \cite[Ch. I, Props.\ 1.1--1.2]{Is96} 
whose proof was not included therein, so we give the sketch of a proof here, 
for the sake of completeness. 

\begin{prop}\label{quasi_prop}
Assume the setting of Definition~\ref{quasi_def}. 
Then the following assertions hold: 
\begin{enumerate}[(i)]
\item\label{quasi_prop_item1} 
The map $\pi_a\colon S\to \Bc(\Hc)$ is a unitary representation. 
\item\label{quasi_prop_item2} 
If the representation $\pi_a$ is irreducible, then the action of $S$ on $X$ is ergodic. 
\item\label{quasi_prop_item3}  
If $S$ is a topological group and  one has  
$$\lim\limits_{s\to\1}\mu(E\, \triangle\, (E.s))
=\lim\limits_{s\to\1}\int\limits_{E\, \triangle\, (E.s)}\vert j(\cdot,s)^{1/2}-1\vert^2\de\mu
=\lim\limits_{s\to\1}\int\limits_E (a(s,\cdot)-1)\de\mu=0$$ 
for every measurable set $E\subseteq X$ with $\mu(E)<\infty$, 
then the representation $\pi_a$ is continuous. 
\end{enumerate}
\end{prop}

\begin{proof} 
Assertion~\eqref{quasi_prop_item1} is based on a straightforward computation.  

For Assertion~\eqref{quasi_prop_item2} note that 
for every measurable set $E\subseteq X$ which is $G$-invariant, 
the multiplication operator $M_{\chi_E}\in\Bc(\Hc)$ is an orthogonal projection 
whose image is invariant under $\pi_a(s)$ for all $s\in S$. 

For Assertion~\eqref{quasi_prop_item3} we use that the values of $\pi_a$ are unitary operators on~$\Hc$, 
hence an $(\varepsilon/3)$-argument shows that it suffices to check that 
$\lim\limits_{s\to\1}\Vert\pi_a(s)\varphi-\varphi\Vert=0$ for $\varphi$ 
in some subset of $\Hc$ that spans a dense linear subspace. 
Using the assumptions, one can check that $\lim\limits_{s\to\1}\Vert\pi_a(s)\chi_E-\chi_E\Vert=0$ 
for every measurable set $E\subseteq X$ with $\mu(E)<\infty$, and this completes the proof. 
\end{proof}

For the following theorem we recall that a \emph{multiplicity-free representation} 
is a unitary representation whose commutant is commutative. 

\begin{thm}\label{quasi_th}
Assume the setting of Definition~\ref{quasi_def}, where $(X, \mu)$ satisfies either of the conditions in Lemma~\ref{masa},   
and let
$$S_0:=\{s\in S\mid x.s=x \text{ for a.e. } x\in X\}.$$  
If the set $\{a(\cdot,s)\mid s\in S_0\}$ spans a $w^*$-dense linear subspace of $L^\infty(X,\mu)$, 
then $\pi_a\colon S\to \Bc(\Hc)$ is a multiplicity-free representation and moreover the following assertions are equivalent: 
\begin{enumerate}[(i)]
\item The action of $S$ on $(X,\mu)$ is ergodic. 
\item The representation $\pi_a$ is irreducible. 
\end{enumerate}
\end{thm}

\begin{proof}
Recall that $\Hc=L^2(X,\mu)$ and for any $\psi\in L^\infty(X,\mu)$ we denote by 
$M_\psi\in\Bc(\Hc)$ be the operator of  multiplication by~$\psi$. 
By Lemma~\ref{masa}, the operator algebra 
$$\Ac:=\{M_\psi\mid \psi\in L^\infty(X,\mu)\}\subseteq\Bc(\Hc)$$ 
is a maximal self-adjoint subalgebra of $\Bc(\Hc)$. 
To prove that $\pi_a$ is a multipli\-ci\-ty-free representation, we will show that $\pi_a(S)'\subseteq\Ac$. 
Hence we must prove that if $T\in\Bc(\Hc)$ and $T\pi(s)=\pi(s)T$ for all $s\in S$, then $T\in\Ac$. 
In fact we will prove a stronger fact, 
namely if $T\in\Bc(\Hc)$ and $T\pi(s)=\pi(s)T$ for all $s\in S_0$, then $T\in\Ac$. 

If $s\in S_0$, then it is clear that $j(x,s)=1$ and $\varphi(x.s)=\varphi(x)$ for a.e. $x\in X$, 
where $\varphi\in L^2(X,\mu)$ is arbitrary, 
and it then follows by the definition of $\pi_a$ that 
$\pi_a(s)$ is the operator of multiplication by $a(\cdot,s)\in L^\infty(X,\mu)$. 
Since the set $\{a(\cdot,s)\mid s\in S_0\}$ spans a $w^*$-dense linear subspace of $L^\infty(X,\mu)$ 
by hypothesis, 
it then follows that if $T\in\Bc(\Hc)$ and $T\pi(s)=\pi(s)T$ for all $s\in S_0$, then $T\in\Ac'$. 
We have seen above that $\Ac$ is a maximal self-adjoint subalgebra of $\Bc(\Hc)$, hence $\Ac'=\Ac$, 
and then $T\in\Ac$, as claimed above. 
This completes the proof of the fact that $\pi_a$ is a multiplicity-free representation. 

Moreover, if the representation~$\pi_a$ is irreducible, then the action of $S$ on $(X,\mu)$ is ergodic 
by Proposition~\ref{quasi_prop}\eqref{quasi_prop_item2}. 
Conversely, let us assume that the action of $S$ on $(X,\mu)$ is ergodic. 
In order to prove that the representation $\pi_a$ is irreducible, 
we must show that if $T\in\Bc(\Hc)$ satisfies $T\pi(s)=\pi(s)T$ for all $s\in S$, 
then $T$ is a scalar multiple of the identity operator on~$\Hc$. 
In fact, using the condition $T\pi(s)=\pi(s)T$ for all $s\in S_0$, 
we obtain by the above reasoning that $T=M_\psi$ for some $\psi\in L^\infty(X,\mu)$. 
Then for all $s\in S$ and $\varphi\in L^\infty(X,\mu)$ one has 
$$\begin{aligned}
\psi(x)j(x,s)^{1/2}\varphi(x.s)
&=(M_\psi\pi_a(s)\varphi)(x) \\
&=(T\pi_a(s)\varphi)(x) \\
&=(\pi_a(s)T\varphi)(x) \\
&=(\pi_a(s)M_\psi\varphi)(x) \\
&=j(x,s)^{1/2}\psi(x.s)\varphi(x.s)
\end{aligned}$$
for a.e. $x\in X$. 
This implies that for all $s\in S$ one has 
$\psi(x)=\psi(x.s)$ for a.e. $x\in X$. 
Since $\psi\in L^\infty(X,\mu)$ and the action of $S$ on $(X,\mu)$ is ergodic, 
it then follows that $\psi$ is constant a.e. on $X$, hence the multiplication operator 
$T=M_\psi$ is a scalar multiplication of the identity operator on~$\Hc$, 
and this completes the proof. 
\end{proof}

\begin{rem}
\normalfont
As we will see in the examples presented in the following sections of this paper, 
the group $S_0$ from Theorem~\ref{quasi_th} is an abstract version of 
the Lie subgroup that corresponds to a polarization of a nilpotent Lie algebra. 
More precisely, one can interpret the representation $\pi_a\colon S\to \Bc(L^2(X,\mu))$ 
as the representation induced from the character $\chi_0\colon S_0\to\TT$, $\chi_0(s):=a(x_0,s)$, 
for some fixed $x_0\in X$ (if any)  with $x_0.s=x_0$ for all $s\in S_0$. 
It is worth noting that if there exists such a point $x_0\in X$, 
then the above $\chi_0$ is a group homomorphism because of the cocycle properties of $a$.
\end{rem}

\section{Applications to group actions on locally compact spaces}

The following proposition establishes irreducibility of some unitary representations that 
play a very significant role in \cite{BB09} and \cite{BB10}. 
See Examples~\ref{Ex09}--\ref{Ex10} below for more specific information in this connection. 

\begin{prop}\label{irred_prop}
Let $G$ be a group and $(X,\mu)$ be any locally compact space endowed with a Radon measure. 
Assume that $\alpha\colon G\times X\to X$, $(g,x)\mapsto\alpha_g(x)$, 
is an action of $G$ on $X$ by measure-preserving transformations. 
Let $\Fc$ be any $G$-invariant vector space of real measurable functions on $X$, 
with the corresponding representation $\lambda\colon G\to\End(\Fc)$, $\lambda_g(f):=f\circ\alpha_{g^{-1}}$. 
Assume in addition that 
the linear span of the set $\{\exp(\ie f)\mid f\in\Fc\}$ is $w^*$-dense in $L^\infty(X,\mu)$. 
Then the following conditions are equivalent:
\begin{enumerate}
\item\label{irred_prop_item1} The action $\alpha$ is ergodic. 
\item\label{irred_prop_item2} The unitary representation 
$$\pi\colon \Fc\rtimes_\lambda G\to\Bc(L^2(X,\mu)),\quad 
(\pi(f,g)\varphi)(x)=\ee^{\ie f(x)}\varphi(\alpha_{g^{-1}}(x))$$ 
is irreducible. 
\end{enumerate}
\end{prop}

\begin{proof}
This is just a special case of Theorem~\ref{quasi_th},
with $S=\Fc\rtimes_\lambda G$.
Indeed in this case  $S_0=(\Fc,+)$, and it acts trivially on $X$. 
The fact that group action of  $\Fc\rtimes_\lambda G$ on $(X,\mu)$ is ergodic is  
equivalent to ergodicity of the action of $G$ on  $(X,\mu)$, since the action of $\Fc$ on $(X,\mu)$ is trivial. 
\end{proof}

\begin{rem}\label{irred_rem}
\normalfont
In Proposition~\ref{irred_prop} the ergodicity hypothesis is necessary for the representation~$\pi$ to be irreducible, without imposing any condition on the linear span of the set $\{\exp(\ie f)\mid f\in\Fc\}$. 
This follows by Proposition~\ref{quasi_prop}\eqref{quasi_prop_item2}.
\end{rem}

For the transitive group action of a connected simply connected nilpotent Lie group on itself by left translations, 
the following corollary implies that the unitary representations constructed in \cite[Subsect. 2.4]{BB09} 
are irreducible. 
Using suitable global coordinates on coadjoint orbits of nilpotent Lie groups 
and the transitivity of coadjoint action on its orbits, 
this corollary also recovers the result of \cite[Prop. 5.1(2)]{BB10}. 
See Examples \ref{Ex09}--\ref{Ex10} below for more details in this connection. 

\begin{cor}\label{irred_cor}
Let $G$ be a group and $X$ be a finite-dimensional real vector space with a Lebesgue measure~$\mu$. 
Assume that $\alpha\colon G\times X\to X$, $(g,x)\mapsto\alpha_g(x)$, 
is an action of $G$ on $X$ by measure-preserving transformations. 
Let $\Fc$ be any $G$-invariant vector space of real measurable functions on $X$, 
with the corresponding representation $\lambda\colon G\to\End(\Fc)$, $\lambda_g(f):=f\circ\alpha_{g^{-1}}$, 
and define  
the unitary representation 
$$\pi\colon \Fc\rtimes_\lambda G\to\Bc(L^2(X,\mu)),\quad 
(\pi(f,g)\varphi)(x)=\ee^{\ie f(x)}\varphi(\alpha_{g^{-1}}(x)).$$ 
If the linear dual space of $X$ satisfies $X^*\subseteq\Fc$, then the following conditions are equivalent:  
\begin{enumerate}[(i)]
\item The action $\alpha$ is ergodic. 
\item The representation $\pi$ is irreducible. 
\end{enumerate}
\end{cor}

\begin{proof}
If the representation $\pi$ is irreducible, then $\alpha$ is ergodic by Remark~\ref{irred_rem}. 

Conversely, the result will follow by Proposition~\ref{irred_prop} as soon as we will have proved that 
the linear span of the set $\{\exp(\ie \xi)\mid \xi\in X^*\}$ is $w^*$-dense in $L^\infty(X,\mu)$. 
To check this, recall that the predual of the von Neumann algebra $L^\infty(X,\mu)$ 
is $L^1(X,\mu)$ and the corresponding duality pairing is 
$$L^\infty(X,\mu)\times L^1(X,\mu)\to\CC,\quad 
(\varphi,\psi)\mapsto\langle\varphi,\psi\rangle:=\int\limits_X\varphi\psi\de\mu.$$
On the other hand, if $\psi\in L^1(X,\mu)$ and 
$0=\langle\exp(\ie \xi),\psi\rangle=\int\limits_X\exp(\ie \xi)\psi\de\mu$ for all $\xi\in X^*$, 
then $\psi=0$ by the injectivity property of the Fourier transform. 
It then follows by the Hahn-Banach theorem that indeed 
the linear span of the set $\{\exp(\ie \xi)\mid \xi\in X^*\}$ is $w^*$-dense in $L^\infty(X,\mu)$, 
and this completes the proof. 
\end{proof}

\begin{ex}[{\cite[Subsect. 2.4]{BB09}}]\label{Ex09}
\normalfont
Let $G$ be any connected, simply connected, nilpotent Lie group with some fixed left invariant Haar measure, 
and $\Fc\subseteq\Ci(G)$ be a linear subspace 
of the space of smooth functions on~$G$. 
satisfying the following conditions: 
\begin{enumerate}
\item\label{orbit0_item1}
The linear space $\Fc$ is invariant under the representation of $G$ by left translations, 
$\lambda\colon G\to\End(\Ci(G))$, $(\lambda_g\phi)(x)=\phi(g^{-1}x)$. 
We denote again by $\lambda\colon G\to\End(\Fc)$ the restriction to $\Fc$ 
of the above representation $\lambda$ of~$G$. 
\item\label{orbit0_item3}
The mapping $G\times\Fc\to\Fc$, $(g,\phi)\mapsto\lambda_g\phi$ is continuous. 
\item\label{orbit0_item5} 
We have $\gg^*\subseteq\{\phi\circ\exp_G\mid\phi\in\Fc\}$. 
\end{enumerate}
We define 
$\pi\colon \Fc\rtimes G\to \Bc(L^2(G))$ by 
$$(\pi(\phi,g)f)(x)
=\ee^{\ie\phi(x)}f(g^{-1}x)
$$ 
for all $\phi\in\Fc$, $g\in G$, and $f\in L^2(G)$, and almost all $x\in G$. 

Hence $\pi$ is as in Proposition~\ref{irred_prop}. 
In order to apply that proposition we must check that 
the linear span of the set $\{\exp(\ie \phi)\mid \phi\in\Fc\}$ is $w^*$-dense in $L^\infty(G)$, 
hence that if $\psi\in L^1(G)$ and $\int\limits_G\psi \exp(\ie \phi)=0$ for all $\phi\in\Fc$, 
then necessarily $\psi=0$. 
To this end, using the above condition for $\phi=\xi\circ\log_G$ with arbitrary $\xi\in\gg^*$ 
(note that $\phi\in\Fc$ by the hypothesis \ref{orbit0_item5}), 
we obtain that the Fourier transform of $\psi$ is zero, hence $\psi=0$. 
Finally, the right action of $\Fc\rtimes G$ on $G$ given by 
$$G\times(\Fc\rtimes G)\to G,\quad (x,(g,\phi))\mapsto g^{-1}x$$
is transitive, hence ergodic (see Remark~\ref{erg_rem}), 
and then by Proposition~\ref{irred_prop} the representation $\pi$ is irreducible. 

Let us also note that the above hypotheses on $\Fc$ ensure that $\Fc$ is a admissible function space 
in the sense of \cite[Def. 2.8]{BB09}. 
\end{ex}

\begin{ex}[{\cite[Prop. 5.1(2)]{BB10}}]\label{Ex10}
\normalfont
Let $G$ be any connected, simply connected, nilpotent Lie group, with its center $Z$ and 
the corresponding Lie algebras ${\mathfrak z}\subseteq{\mathfrak g}$. 
Endow the coadjoint orbit~${\mathcal O}$ with its Liouville measure and define 
$$\widetilde{\pi}\colon G\ltimes_{{\rm Ad}}{\mathfrak g}\to{\mathcal B}(L^2({\mathcal O})),\quad 
(\widetilde{\pi}(g,Y)f)(\xi)={\rm e}^{{\rm i} {\langle\xi,Y\rangle}}f({\rm Ad}^*_G(g^{-1})\xi).$$
Then the following assertions hold: 
\begin{enumerate}[(i)]
\item\label{double_item1} 
The group $\widetilde{G}:=G\ltimes_{{\rm Ad}}{\mathfrak g}$ 
is nilpotent and its center is $Z\times{\mathfrak z}$. 
\item\label{double_item2} 
$\widetilde{\pi}$ is a unitary irreducible representation 
of $\widetilde{G}$. 
\end{enumerate}

We recall that the multiplication in the semi-direct product group $\widetilde{G}$ 
is given by 
\begin{equation}\label{double_eq00}
(g_1,Y_1)\cdot(g_2,Y_2)=(g_1g_2,Y_1+{\rm Ad}_G(g_1)Y_2)
\end{equation}
and the bracket in the corresponding Lie algebra 
$\widetilde{{\mathfrak g}}={\mathfrak g}\ltimes_{{\rm ad}}{\mathfrak g}$ is 
defined by 
\begin{equation*}
[(X_1,Y_1),(X_2,Y_2)]=([X_1,X_2],[X_1,Y_2]-[X_2,Y_1]).
\end{equation*}
An inspection of these equations shows that $\widetilde{{\mathfrak g}}$ 
is a nilpotent Lie algebra with its center ${\mathfrak z}\times{\mathfrak z}$. 

To see that $\widetilde{\pi}$ is a representation 
we need to check that the function 
$$ a\colon {\mathcal O}\times \widetilde{G}\to \TT,\quad 
a(\xi,(g,Y)):={\rm e}^{{\rm i} {\langle\xi,Y\rangle}}$$
is a cocycle in the sense of Definition~\ref{quasi_def}. 
In fact, using the right action of $\widetilde{G}$ on $\Oc$ given by 
\begin{equation}\label{double_eq01}
\Oc\times\widetilde{G}\to\Oc,\quad 
(\xi,(g,Y))\mapsto \xi\circ \Ad_G(g)=\Ad_G^*(g^{-1})\xi, 
\end{equation} 
it follows by \eqref{double_eq00} and the above definition of $a$ that 
$$\begin{aligned}
a(\xi,(g_1,Y_1)(g_2,Y_2))
&=a(\xi,(g_1g_2,Y_1+{\rm Ad}_G(g_1)Y_2)) \\
&= {\rm e}^{{\rm i}{\langle\xi,Y_1+{\rm Ad}_G(g_1)Y_2\rangle}} \\
&={\rm e}^{{\rm i}{\langle \xi,Y_1\rangle}}
{\rm e}^{{\rm i}{\langle\xi\circ\Ad_G(g_1),Y_2\rangle}} \\
&=a(\xi,(g_1,Y_1))a(\xi.(g_1,Y_1),(g_2,Y_2)).
\end{aligned}$$
The property $a(\xi,\1)=\xi$ for al $\xi\in\Oc$ is clear from the definition of $a$. 
Also note that 
the Liouville measure on ${\mathcal O}$ is invariant under the group action~\eqref{double_eq01}. 
It then follows by Proposition~\ref{quasi_prop} that $\widetilde{\pi}$ is a continuous unitary representation. 

Moreover, to see that $\widetilde{\pi}$ is irreducible we will use Corollary~\ref{irred_cor}.  
To this end, first note that the group action~\eqref{double_eq01} is transitive, hence ergodic (see Remark~\ref{erg_rem}). 
Furthermore, recall that the mapping 
$${\mathcal O}\to\gg_e^*,\quad 
\xi\to\xi\vert_{\gg_e} $$
is a global chart which takes the Liouville measure of ${\mathcal O}$ 
to a Lebesgue measure on~$\gg_e^*$, 
where $e$ is the jump index set of $\Oc$ with respect to some Jordan-H\"older basis in~$\gg$  
(see for instance \cite{BB10}). 
Then we can use the Fourier transform to see that the linear subspace generated by $\{{\rm e}^{{\rm i}{\langle Y,\cdot\rangle}}\mid Y\in{\mathfrak g}\}$ 
is weak$^*$-dense in $L^\infty({\mathcal O})$ ($\simeq L^1({\mathcal O})^*$). 
Therefore we can use Corollary~\ref{irred_cor} to obtain that $\widetilde{\pi}$ is irreducible.

In addition to the above properties of $\widetilde{\pi}$ we also recall some additional information 
on the irreducible representation $\widetilde{\pi}$ that was obtained in \cite[Prop. 5.1(2)]{BB10}. 
Firstly, the space of smooth vectors for the representation $\widetilde{\pi}$ 
is ${\mathcal S}({\mathcal O})$. 
Moreover, 
select any Jordan-H\"older basis $X_1,\dots,X_n$ in $\gg$ 
and define 
$$\widetilde{X}_j=
\begin{cases}
\hfill (0,X_j) &\text{ for }j=1,\dots,n,\\
(X_{j-n},0) &\text{ for }j=n+1,\dots,2n.
\end{cases} 
$$
Then $\widetilde{X}_1,\dots,\widetilde{X}_{2n}$ is a Jordan-H\"older basis 
in $\widetilde{{\mathfrak g}}$ and the corresponding predual for 
the coadjoint orbit $\widetilde{{\mathcal O}}\subseteq\widetilde{{\mathfrak g}}^*$ associated with 
the representation $\widetilde{\pi}$ is $$\widetilde{{\mathfrak g}}_{\widetilde{e}}={\mathfrak g}_e\times{\mathfrak g}_e\subseteq\widetilde{{\mathfrak g}},$$
where $\widetilde{e}$ is the set of jump indices for~$\widetilde{{\mathcal O}}$. 
\end{ex}

\section{Application to Gaussian measures on Hilbert spaces} 

We first recall here a few facts from \cite{BB10c} and \cite{BBM15}. 

\begin{defn}\label{heisenberg}
\normalfont
If $\Vc$ is a real Hilbert space, $A\in\Bc(\Vc)$ with 
$(Ax\mid y)=(x\mid Ay)$ for all $x,y\in\Vc$, 
and moreover $\Ker A=\{0\}$, then  
the \emph{Heisenberg algebra} associated with the pair $(\Vc,A)$ is  
the real Hilbert space $\hg(\Vc,A)=\Vc\dotplus\Vc\dotplus\RR$ endowed with the Lie bracket  
defined by 
$[(x_1,y_1,t_1),(x_2,y_2,t_2)]=(0,0,(Ax_1\mid y_2)-(Ax_2\mid y_1))$.   
The corresponding \emph{Heisenberg group} $\HH(\Vc,A)=(\hg(\Vc,A),\ast)$ 
is the Lie group whose underlying manifold is $\hg(\Vc,A)$ and 
whose multiplication is defined by 
$$(x_1,y_1,t_1)\ast(x_2,y_2,t_2)=
(x_1+x_2,y_1+y_2,t_1+t_2+((Ax_1\mid y_2)-(Ax_2\mid y_1))/2) $$
for $(x_1,y_1,t_1),(x_2,y_2,t_2)\in\HH(\Vc,A)$. 
\end{defn}

Let $\Vc_{-}$ be a real Hil\-bert space and $(\cdot\mid\cdot)_{-}$ be its scalar product. 
For every $a\in\Vc_{-}$ and 
every symmetric, nonnegative, injective, \emph{trace-class} operator $K$ on $\Vc_{-}$ 
there is a unique probability Borel measure~$\gamma$ on $\Vc_{-}$
with  
$$(\forall x\in\Vc_{-})\quad 
\int\limits_{\Vc_{-}}\ee^{\ie(x\mid y)_{-}}\de\gamma(y)
=\ee^{\ie(a\mid x)_{-}-\frac{1}{2}(Kx\mid x)_{-}} $$
and $\gamma$ is called the  
\emph{Gaussian measure with the mean $a$ and the variance $K$}.  

Now assume that 
$a=0$  
and let $\Vc_{+}:=\Ran K$ and $\Vc_0:=\Ran K^{1/2}$ be endowed 
with the scalar products given by $(Kx\mid Ky)_{+}:=(x\mid y)_{-}$ and 
$(K^{1/2}x\mid K^{1/2}y)_0:=(x\mid y)_{-}$, respectively, for all $x,y\in\Vc_{-}$, 
which turn the linear bijections 
$K\colon\Vc_{-}\to\Vc_{+}$ and $K^{1/2}\colon\Vc_{-}\to\Vc_0$ 
into isometries. 
We thus obtain the real Hilbert spaces
$$\Vc_{+}\hookrightarrow\Vc_0\hookrightarrow\Vc_{-} $$
where the inclusion maps are Hilbert-Schmidt operators, 
since $K^{1/2}\in\Bc(\Vc_{-})$ is a Hilbert-Schmidt operator. 
Also, the scalar product of $\Vc_0$ extends to a duality pairing $(\cdot\mid\cdot)_0\colon\Vc_{-}\times\Vc_{+}\to\RR$. 

We also recall that for every $x\in\Vc_{+}$ the translated measure $\de\gamma(-x+\cdot)$ 
is absolutely continuous with respect to $\de\gamma(\cdot)$ 
and we have the Cameron-Martin formula
$$\de\gamma(-x+\cdot)=\rho_x(\cdot)\de\gamma(\cdot) 
\quad\text{ with }\rho_x(\cdot)=\ee^{(\cdot\mid x)_0-\frac{1}{2}(x\mid x)_0}.$$  

\begin{defn}\label{sch_def}
\normalfont 
Let $\Vc_{+}$ be a real Hilbert space with the scalar product denoted by $(x,y)\mapsto (x\mid y)_{+}$. 
Also let $A\colon\Vc_{+}\to\Vc_{+}$ be a nonnegative, symmetric, injective, trace-class operator.  
Let $\Vc_0$ and $\Vc_{-}$ be the completions of $\Vc_{+}$ with respect to the scalar products  
$$(x,y)\mapsto (x\mid y)_0:=(A^{1/2}x\mid A^{1/2}y)_{+}$$ 
and 
$$(x,y)\mapsto (x\mid y)_{-}:=(Ax\mid Ay)_{+},$$ 
respectively.  
Then the operator $A$ uniquely extends to a nonnegative, symmetric, injective, trace-class operator $K\in\Bc(\Vc_{-})$,   
hence by the above observations one obtains the Gaussian measure $\gamma$ on $\Vc_{-}$ with variance $K$ and mean~$0$. 

One can also construct the Heisenberg group $\HH(\Vc_{+},A)$. 
The \emph{Schr\"odinger representation} 
$\pi\colon\HH(\Vc_{+},A)\to\Bc(L^2(\Vc_{-},\gamma))$ is defined by 
$$\pi(x,y,t)\phi=\rho_x(\cdot)^{1/2}
\ee^{\ie(t+(\cdot\mid y)_0+\frac{1}{2}(x\mid y)_0)}\phi(-x+\cdot) $$
for $(x,y,t)\in\HH(\Vc_{+},A)$ and $\phi\in L^2(\Vc_{-},\gamma)$. 
\end{defn}

\begin{prop}
\label{sch_irred}
The representation 
$\pi\colon\HH(\Vc_{+},A)\to\Bc(L^2(\Vc_{-},\gamma))$ 
from Definition~\ref{sch_def} is irreducible. 
\end{prop} 

\begin{proof}
See for instance from \cite[Rem. 3.6]{BB10c} or \cite{BBM15}.
\end{proof}

\begin{cor}
In the above setting, the action by translations of $\Vc_{+}$ on $(\Vc_{-},\gamma)$ is ergodic. 
\end{cor}

\begin{proof} 
In the present framework, 
the representation $\pi$ is the unitary representation associated to the measure space $(\Vc_{-},\gamma)$ 
acted on by the additive group $(\Vc_{+},+)$ by translations.  
The cocycle of that measurable dynamical system which gives rise to the representation $\pi$ is given by 
$$a(\cdot,(x,y,t))=\ee^{\ie(t+(\cdot\mid y)_0+\frac{1}{2}(x\mid y)_0)}$$
for all $(x,y,t)\in\HH(\Vc_{+},A)$. 
The conclusion follows by Propositions \ref{sch_irred} and \ref{quasi_prop}\eqref{quasi_prop_item2} 
for the right group action 
$$\Vc_{-}\times\HH(\Vc_{+},A)\to \Vc_{-},\quad (v,(x,y,t))\mapsto -x+v$$
and we are done.  
\end{proof}




\begin{thebibliography}{1000000}

\bibitem[BB09]{BB09}
I.~Belti\c t\u a, D.~Belti\c t\u a, 
\textit{Magnetic pseudo-differential Weyl calculus on nilpotent Lie groups}. 
Ann. Global Anal. Geom. 36 (2009), no. 3, 293--322. 

\bibitem[BB10a]{BB10}
I.~Belti\c t\u a, D.~Belti\c t\u a, 
\textit{Smooth vectors and Weyl-Pedersen calculus for representations of nilpotent Lie groups}. 
Ann. Univ. Buchar. Math. Ser. 1(LIX) (2010), no. 1, 17--46.

\bibitem[BB10b]{BB10c}
I. Belti\c t\u a, D. Belti\c t\u a, 
\textit{On Weyl calculus in infinitely many variables}.  
In: P. Kielanowski, V. Buchstaber, A. Odzijewicz, M. Schlichenmaier, Th. Voronov (eds.), 
\textit{XXIX Workshop on Geometric Methods in Physics}, 
AIP Conf. Proc., 1307,  Amer. Inst. Phys., Melville, NY, 2010, pp. 19--26.

\bibitem[BBM15]{BBM15}
I.~Belti\c t\u a, D.~Belti\c t\u a, M. M\u antoiu, 
\textit{On Wigner transforms in infinite dimensions}. Preprint arXiv:1501.05404 [math.RT]. 

\bibitem[Di69]{Di69}
Dixmier, J., 
{\it Les algebres d'op\'erateurs dans l'espace hilbertien (alg\`ebres de von Neumann)}. 
Deuxi\`eme \'edition, revue et augment\'ee. Cahiers Scientifiques, Fasc. XXV. Gauthier-Villars Editeur, Paris, 1969.

\bibitem[Is96]{Is96}
R.S.~Ismagilov, 
{\it Representations of infinite-dimensional groups}. 
Translations of Mathematical Monographs, 152. American Mathematical Society, Providence, RI, 1996.

\bibitem[SS08]{SS08}
A.M.~Sinclair, R.R.~Smith, 
{\it Finite von Neumann algebras and masas}. 
London Mathematical Society Lecture Note Series, 351. Cambridge University Press, Cambridge, 2008.

\bibitem[Ta03]{Ta03}
Takesaki, M., 
{\it Theory of operator algebras. III}. 
Encyclopaedia of Mathematical Sciences, 127. 
Operator Algebras and Non-commutative Geometry, 8. Springer-Verlag, Berlin, 2003.

\end{thebibliography}
\end{document}